\documentclass[12pt]{amsart}
\usepackage{amsmath}
\usepackage{amsthm}
\usepackage{amsfonts}
\usepackage{amssymb}
\usepackage{dsfont}
\usepackage{xr}

\newlength{\abstand}
\DeclareFontFamily{OT1}{pzc}{}
\DeclareFontShape{OT1}{pzc}{m}{it}{<-> s * [1.150] pzcmi7t}{}
\DeclareMathAlphabet{\mathcal}{OT1}{pzc}{m}{it}

\def\F{{\mathds F}}

\def\R{{\mathds R}}

\def\C{{\mathds C}}
\def\N{{\mathds N}}
\def\Nn{\N_0}
\def\Np{\N_+}
\def\Q{{\mathds Q}}
\def\Z{{\mathds Z}}
\def\P{{\mathds P}}
\def\G{{\mathds G}}
\def\Ns{\str{\N}}
\def\Nns{\str{\Nn}}

\def\a{{\mathfrak a}}
\def\RR{{\bf R}}
\def\PP{{\bf P}}

\def\p{{\mathfrak p}}

\def\fX{{\mathfrak X}}
\def\fU{{\mathfrak U}}

\def\cF{\mathcal F}
\def\cG{\mathcal G}
\def\cH{\mathcal H}
\def\cI{\mathcal I}

\def\cO{\mathcal O}
\def\cP{\mathcal P}
\def\cp{\mathcal p}

\def\cR{\mathcal R}

\def\scR{\str{\cR}}

\def\SCH{\mathcal{Sch}}
\def\Sm{\mathcal{Sm}}
\def\SCHFP{\SCH^{\mathrm{fp}}}
\def\ALGSPC{\mathcal{AlgSpc}}
\def\ALGSPCFP{\ALGSPC^{fp}}

\def\Shv{\mathcal{Shv}}
\def\et{{\mathrm{\acute{e}t}}}
\def\Shvet{\Shv_\et}

\def\im{\mathrm{im}}

\newcommand\spec[1]{{\mathrm{Spec}\,(#1)}}

\newcommand\str[1]{{\mbox{}^*#1}}

\newcommand\Schfp[1]{\SCH^{\mathrm{fp}}_{#1}}
\newcommand\AlgSpcfp[1]{\ALGSPC^{\mathrm{fp}}_{#1}}

\newcommand\op[1]{{#1}^{\mathrm{op}}}

\newcommand\s[1]{N\,#1}

\newcommand\T[1]{T\,#1}

\newcommand\Hom[3]{\mathrm{Hom}_{#1}(#2,#3)}

\newcommand\Pic[1]{\mathrm{Pic}(#1)}
\newcommand\sPic[1]{\str{\Pic{#1}}}

\newcommand\kernel[1]{\mathrm{Ker}\,\left(#1\right)}
\newcommand\coker[1]{\mathrm{Coker}\,\left(#1\right)}

\newcommand\length[2]{l_{#1}\left(#2\right)}
\newcommand\slength[2]{\str{l_{#1}}\left(#2\right)}
\newcommand\ch[1]{\mathrm{char}(#1)}
\newcommand\sch[1]{\str{\mathrm{char}}(#1)}

\swapnumbers
\theoremstyle{definition}
\newtheorem{defi}{Definition}[section]

\newtheorem{bspe}[defi]{Examples}
\newtheorem{satzdefi}[defi]{Proposition/ Definition}
\newtheorem{lemma}[defi]{Lemma}

\newtheorem{bem}[defi]{Remark}
\newtheorem{satz}[defi]{Proposition}
\newtheorem{thm}[defi]{Theorem}
\newtheorem{cor}[defi]{Corollary}
\newtheorem{conj}[defi]{Conjecture}

\input xy
\xyoption{all}

\title{\'Etale and motivic cohomology and ultraproducts of schemes}
\author{Lars Br\"unjes, Christian Serp\'e}
\date{\today}
\email{lbrunjes@gmx.de}
\address{
  Christian Serp\'e \\
  Westf\"alische Wilhelms-Universit\"at M\"unster \\
  Mathematisches Institut \\
  Sonderforschungsbereich 478 ``Geometrische Strukturen in der Mathematik'' \\
  Hittorfstr. 27 \\
  D-48149 M\"unster \\
  Germany}
\email{serpe@uni-muenster.de}

\date{\today}

\begin{document}

\begin{abstract}
This paper is a continuation of the authors article \cite{enlsch}. We
mainly study the behaviour of \'etale cohomology, algebraic cycles
and motives under ultraproducts respectively enlargements. The main motivation for that is to
find methods to transfer statements about \'etale cohomology and algebraic cycles from
characteristic zero to positive characteristic and vice versa. We give
one application to the independence of $l$ of Betti numbers in \'etale
cohomology and applications to the complexity of algebraic cycles.
\end{abstract}

\maketitle

\tableofcontents

\section{Introduction}
Let $\{R_i\}_{i\in I}$ be a family of commutative rings and consider the
usual product of rings $\prod_{i\in I} R_i$. Most properties of rings
do not behave well under this construction. For example, even if all
rings $R_i$ are fields, the product is full of zero divisors. The
situation changes if we choose an ultrafilter $\fU\subset\cP(I)$ and
consider the ultraproduct
$$\prod_{i\in I,\fU}R_i:=\prod_{i\in I} R_i/\sim,$$
where $\sim$ is defined by
$(r_i)_{i\in I}\sim (r'_i)_{i\in I} :\Leftrightarrow \{i\in I\vert
r_i=r'_i\}\in\fU.$ For example, in this situation $\prod_{i\in
  I,\fU}R_i$ is even a field if all $R_i$ are fields. In the
case $I=\N$ and $R_i=\R$, Robinson used this methods to construct an
enlargement $\str{\R}:=\prod_{\N,\fU} \R$ of $\R$ where one can do calculus
with infinitesimals, leading to the area of mathematics known today as
\emph{Nonstandard Analysis}. In the case
where $I=\P$ is the set of prime numbers and $R_p=\F_p$ are the finite
prime fields, the ultraproduct $\prod_{p\in\P,\fU}\F_p$ is an interesting
field for algebraic geometry. Namely, on the one hand, that field
behaves like a finite field, because it is the ultraproduct of finite
fields. But on the other hand, it is a field of
characteristic zero. In order to use this ambiguity in algebraic
geometry, we started in \cite{enlsch} to investigate how schemes behave under
ultraproducts. We constructed and explored the properties of a
functor $N$ which turned a scheme over an ultraproduct of rings
into an ultraproduct of schemes. We described the image of this functor
and constructed a similar functor for coherent sheaves on schemes.
Then we used the ambiguity mentioned above to
give two applications to resolution of singularities and weak
factorisation. For more motivational remarks we refer to the
introduction of \cite{enlsch}.

In the present article we proceed with the investigation we started in
\cite{enlsch}. We study the behavior of \'etale cohomology and algebraic cycles
under ultraproducts. We are mainly interested in a connection between the
\'etale cohomology respectively cycle groups of an (in some sense limited) ultraproduct
of schemes and the ultraproduct of the \'etale cohomology and various cycle groups.
In short, we want to know whether \'etale cohomology and various cycle constructions commute
with ultraproducts. Whereas in \'etale cohomology the results are quite
convincing (cf. e.g. Proposition \ref{satz:etalecentral}), the situation for the
cycle groups is, not surprisingly, much more complicated.

In a forthcoming paper \cite{bscomplc}, we apply our methods to the
question whether a class in the $l$-adic cohomology of a smooth
projective variety over $\Q$, which is algebraic over almost all
finite fields, is also algebraic over $\Q$. We show that this can be
expressed as a question about the uniform complexity of the cycles
representing that class over the finite fields.

\vspace{\abstand}

In the whole article, we do not work with ultraproducts,
but with the enlargement of superstructures as in \cite{enlsch}.
Of course it would be possible to work throughout directly with 
ultraproducts. But from our point of view enlargements provide a 
conceptual way of handling ultraproduct constructions. As the
structures like schemes, Chow groups and \'etale cohomology are quite 
involved we chose this more advanced viewpoint for all our considerations.
For the use of enlargements in category theory, we refer to \cite{enlcat}.

\vspace{\abstand}
Now we describe the content of the paper in a little more detail.

In \cite{enlsch}, we constructed a functor $N$ which assigned
to a scheme $X$ of finite type over an internal ring $R$
a *scheme $N(X)$ over $R$.
In that situation, in the first section, we construct
a functor $N$ from the category of constructible \'etale sheaves on $X$ to the
category of *constructible *\'etale *sheaves on $N(X)$. Then we study the relationship
between the \'etale cohomology of a sheaf $\cF$ and the \'etale
cohomology of $\str{\cF}$.

The independence of $l$ of the Betti numbers of $l$-adic
cohomology is known in characteristic zero, but in general not in
positive characteristic. In the second section, we use the results
of the first section to give an application to the independence of $l$ of Betti numbers
in $l$-adic cohomology. We show that in some sense the Betti numbers of  the $l$-adic 
cohomology is independent of $l$ if the characteristic is large enough. How large depends
on the complexity of the scheme and in some sense on $l$. For a precise statement 
see Theorem \ref{indepofl}.

In the third section, we first construct a functor $N$ for the
triangulated category of Voevodsky motives and then use this to give an
appropriate map $N$ for motivic cohomology. We show that the maps
defined in the first and third section are compatible with each other.

In section four we introduce a notion of complexity of algebraic
cycles and use this to describe the image of the functor $N$ for
cycles. For cycles of codimension one we show that $N$ is bijective
for cycles with finite Hilbert polynomial.
Using the result of Mumford, that Chow groups are not finite dimensional,
we further show that in general $N$ is not injective. We also show how
rational equivalence and the intersection product behave under our
notion of complexity.

In the appendix, we give some lemmas about enlargements in commutative
algebra, which we need in section three.

\section{\'Etale cohomology}
\vspace{\abstand}

In \cite{enlsch}, we explained how a scheme over an ultraproduct of rings
gives rise to an ultraproduct of schemes. Now we want to show that an \'etale sheaf on a
scheme over an ultraproduct gives an ultraproduct of \'etale sheaves
on the ultraproduct of schemes. This is, as in the case of schemes,
only possible for \'etale sheaves which are compact in some sense. For
that we generalize the construction of $N$ for schemes of finite
presentation to algebraic spaces of finite presentation.
After having achieved this, we explore how basic concepts of \'etale
cohomology behave under this construction.

\vspace{\abstand}
We consider the same basic setup as in \cite{enlsch}.
So let $\cR$ be a small subcategory of all commutative rings with:
\begin{itemize}
\item $\Z\in\cR$,
\item for all $A\in\cR$ and all $A$-algebras $B$ of finite
  presentation, we have $B\in\cR$ (up to isomorphism),
\item for all $A\in\cR$ and all prime ideals $\cp \in \spec{A}$,
  the localisation $A_{\cp}$ is in $\cR$ (up to isomorphism).
\end{itemize}

\vspace{\abstand}
As in \cite{enlsch}, let $\Schfp{\cR}$ denote the fibred (over $\cR$) category
of schemes of finite presentation, and let $\AlgSpcfp{\cR}$ denote the
fibred category of algebraic spaces of finite presentation (also fibred over $\cR$).
For general facts about algebraic spaces we refer to \cite{KnAlgSpc}.
We choose a superstructure $\hat{M}$ such that all our small categories are
$\hat{M}$-small,
and let $*:\hat{M}\rightarrow\widehat{\str{M}}$ be an enlargement.
As in the case of schemes, base change along $A\rightarrow \str{A}$ for
$A\in\cR$ defines a functor
$$T: \AlgSpcfp{\cR}\rightarrow \AlgSpcfp{\str{\cR}},$$
such that the diagram
  \[
    \xymatrix{
      {\AlgSpcfp{\cR}} \ar[r]^{\T{}} \ar[d] & {\AlgSpcfp{\scR}} \ar[d] \\
      {\op{\cR}} \ar[r]_{*} & {\op{\scR}} \\
    }
  \]
is commutative.

\vspace{\abstand}

We are looking for a functor
$N:\AlgSpcfp{\str{\cR}}\rightarrow\str{\AlgSpcfp{\str{\cR}}}$ such that
the diagram
\begin{equation}\label{diagram:essentialN}
    \xymatrix{
      {\AlgSpcfp{\cR}} \ar[r]_{\T{}} \ar@(ur,ul)[rr]^{*} \ar[d] &
        {\AlgSpcfp{\scR}} \ar@{.>}[r]_{\s{}} \ar[d] &
        {\str{\AlgSpcfp{\scR}}} \ar[d] \\
      {\op{\cR}} \ar[r]_{*} & {\op{\scR}} \ar@{=}[r] & {\op{\scR}} \\
    }
  \end{equation}
is commutative.

\vspace{\abstand}

\begin{satzdefi}\label{satzconstrN}

There is an essentially unique functor
$$N:\AlgSpcfp{\str{\cR}}\rightarrow\str{\AlgSpcfp{\str{\cR}}}$$
of fibrations over $\str{\cR}$ such that diagram
(\ref{diagram:essentialN}) commutes. 
\end{satzdefi}

\vspace{\abstand}

\begin{bem}
The uniqueness of $N$ from above can be made precise as follows:
$N:\AlgSpcfp{\str{\cR}}\rightarrow\str{\AlgSpcfp{\str{\cR}}}.$
is the right  Kan extension of $*:\AlgSpcfp{\cR}\rightarrow
\str{\AlgSpcfp{\str{\cR}}}$ along
$T:\AlgSpcfp{\str\cR}\rightarrow\str{\AlgSpcfp{\str{\cR}}}$ in the
2-category of fibrations.
\end{bem}

\vspace{\abstand}

The proof of the proposition relies mainly on the following lemma.
\begin{lemma}\label{fundlemmaN}
Let $A=colim_{\lambda\in L}$ be the colimit of a filtered system of
rings $(A_{\lambda})_{\lambda\in L}$.
\begin{enumerate}
\item Let $\lambda_0\in L$, and let $X_{\lambda_0}$ and $Y_{\lambda_0}$ be
  algebraic spaces over $A_{\lambda_0}$. We assume that
  $X_{\lambda_0}$ is quasi compact and that $Y_{\lambda_0}$ is locally
  of finite presentation over $A_{\lambda_0}$. Then the canonical map\\
  \begin{minipage}[c]{15cm}
  \begin{multline*}
  colim_{\lambda>\lambda_0}
  \Hom{A_{\lambda_0}}{X_{\lambda_0}\otimes_{A_{\lambda_0}}A_{\lambda}}
  {Y_{\lambda_0}\otimes_{A_{\lambda_0}}A_{\lambda}} \longrightarrow \\
  \Hom{A_{\lambda_0}}{X_{\lambda_0}\otimes_{A_{\lambda_0}}A_}
  {Y_{\lambda_0}\otimes_{A_{\lambda_0}}A}
  \end{multline*}
  \end{minipage}
  is a bijection.
\item  Let $X$ be an algebraic space of finite presentation over
  $A$. Then there is a $\lambda_0\in L$, an algebraic space $X_0$ of
  finite presentation over $A_{\lambda_0}$ and an isomorphism
  $$ X_0\otimes_{A_{\lambda_0}} A\rightarrow X.$$
\end{enumerate}
\end{lemma}
\begin{proof}[Proof of Lemma \ref{fundlemmaN}]
In \cite{MR1771927}[Proposition 4.18], the analogous statements for algebraic stacks can be found,
and from this, (i) follows immediately.
The second part follows from the analogous statements
for schemes, which can be found in \cite{ega43}[§8].
\end{proof}

\vspace{\abstand}

\begin{proof}[Proof of Proposition \ref{satzconstrN}]
We give an explicit construction of $N$ and leave the rest to the
reader. For an $A\in\str{\cR}$ and $X\in\str{\AlgSpcfp{\str{\cR}}}$ we
can choose by lemma \ref{fundlemmaN}(ii) a subring $A_0\subset A$ of finite type over $\Z$, a
$X_0\in\ALGSPCFP/A_0$, and an isomorphism $\varphi: X_0\otimes_{A_0}
A\rightarrow X$. Because $A_0$ is of finite type over $\Z$, we have a
canonical internal ring homomorphism $\str{A}_0\rightarrow A$ (cf.
\cite{enlsch}[Proposition/Definition 3.2]). With that we define
$$N(X):= \str{X_0}\otimes_{\str{A_0}}A$$
For a different choice $(A_0', X_0',\varphi')$ we have the isomorphism
$$X_0\otimes_{A_0}A\xrightarrow{\varphi'^{-1}\circ \varphi}
X_0'\otimes_{A_0}A.$$
By \ref{fundlemmaN} (i) there is a $B\subset A$ of finite type over
$\Z$ with $A_0,A_0'\subset B$ and an isomorphism
$$\psi_0:X_0\otimes_{A_0}B\xrightarrow{\sim}X_0'\otimes_{A_0'}B$$
such that
\[
\xymatrix{X_0\otimes_{A_0}A\ar[r]^{\varphi'^{-1}\circ \varphi}
  \ar[d]^{\sim} &  X_0'\otimes_{A_0}A \ar[d]^{\sim} \\
  X_0\otimes_{A_0}B_0\otimes_{B_0}A \ar[r]^{\psi_0\otimes_{B_0}A} &
  X_0'\otimes_{A_0'}B_0\otimes_{B_0}A
}\]
So this $\psi_0$ defines an isomorphism
$$\str{X_0}\otimes_{\str{A_0}}A\xrightarrow{\sim}\str{X_0'}\otimes_{\str{A_0}}A.$$
Again by \ref{fundlemmaN} (i) it can be shown that this isomorphism is
independent of the choices of $B_0$ and $\psi$. By similar argument as
in \cite{enlsch}[Theorem 3.4] for schemes,
one shows that $N$ is functorial.
\end{proof}

\vspace{\abstand}

\begin{bem}
In the construction it is important that $A_0$ is not only in $\cR$
but even of finite type over $\Z$. Otherwise there would be no canonical
morphism $\str{A}_0\rightarrow A$.
\end{bem}

\vspace{\abstand}

For a scheme $X$ we consider the \'etale topology on the category
$\SCHFP/X$ and denote the resulting site by $(\SCHFP/X)_\et$. We denote by
$\Shvet(X):=\Shv((\SCHFP/X)_\et)$ the category of sheaves on
$(\SCHFP/X)_\et$.
For a *scheme $X$ we use the
notation $\str{\Shvet(X)}$  for the internal category of *sheaves, and
for $B\in\cR$ and $X\in\SCHFP/B$ there is a canonical functor
$$*:\Shvet(X)\rightarrow \str{\Shv}_\et(\str{X}).$$ For more details
about this we refer to our paper \cite{nsetale}.

\vspace{\abstand}

\begin{bem}
For a quasi compact $X$, e.g. if $X$ is of finite presentation over an
affine scheme, the restriction functor
$$\Shv((\SCH/X)_\et)\rightarrow \Shv((\SCHFP/X)_\et)$$
is an equivalence of categories. So in particular the cohomology on
$X$ in the site $(\SCHFP/X)_\et$ is the same as the usual \'etale
cohomology.
\end{bem}

\vspace{\abstand}

Now let $A\in\str{\cR}$ and $X\in\SCHFP/A$,
and consider the fully faithful Yoneda embedding
$$\ALGSPCFP/X \rightarrow \Shvet(X).$$
We denote by $\Shvet^{fp}(X)$ the  essential image of the above
functor. To define $N$  on $\Shvet^{fp}(X)$, we choose for each
$\cF\in\Shvet^{fp}(X)$ a $Y\in\ALGSPCFP/X$ and an isomorphism
$h_Y\cong\cF$ and define
$$ N(\cF):= \str{h}_{N(Y)},$$
where $\str{h}$ denotes the *Yoneda embedding
$$\str{h}:\str{\ALGSPCFP}/N(X)\rightarrow \str{\Shv}_\et(N(X)).$$
This defines a functor
$$N:\Shvet^{fp}(X)\rightarrow \str{\Shvet(N(X))}.$$

\vspace{\abstand}

\begin{bspe}
We are mainly interested in the following two cases:
\begin{enumerate}
\item Let $\cF:=\G_{m,X}$. Then we have $\cF\in\Shvet^{fp}(X)$, and
  there is an isomorphism $N(\cF)\cong\str{\G}_{m,N(X)}$.
\item Let $\cF:= \mu_n$. Then we have $\cF\in\Shvet^{fp}(X)$, and there is an
  isomorphism $N(\mu_{n,X})\cong \str{\mu}_{n,N(X)}$.
\end{enumerate}
\end{bspe}

\vspace{\abstand}

If $\{U_i\rightarrow X\}$
is a finite \'etale covering of $X$, then $\{N(U_i)\rightarrow N(X)\}$
is an internal *\'etale covering of $N(X)$. We denote by
$$S:\str{\Shv}_\et(N(X))\rightarrow \Shvet(X)$$
the induced functor.

\vspace{\abstand}

The morphisms
$$\Gamma(U,\cF)=\Hom{\ALGSPCFP/X}{U}{Y}\xrightarrow{N}
\Hom{\str{\ALGSPCFP}/N(X)}{N(U)}{N(Y)}$$
define a natural transformation
\begin{equation}\label{nattrans}
\varphi:h\rightarrow S\circ \str{h}\circ N
\end{equation}
from the Yoneda embedding
$$h:\ALGSPCFP/X\rightarrow \Shvet(X)$$
to
$$S\circ \str{h}\circ N:\ALGSPCFP/X\rightarrow \Shvet(X).$$

\vspace{\abstand}

For a *\'etale *sheaf $\cG\in\str{\Shv}_\et(N(X))$, this gives a map
\begin{multline}\label{mapSN}
\Hom{\str{\Shv}_\et(N(X))}{N(h_Y)}{\cG}
=\Hom{\str{\Shv}_\et(N(X))}{[\str{h}N](Y)}{\cG} \\
\xrightarrow{S}\Hom{\Shvet(X)}{[S\str{h}N](Y)}{S(\cG)}
\xrightarrow{\varphi_Y^*}
\Hom{\Shvet(X)}{h_Y}{S(\cG)}.
\end{multline}

\vspace{\abstand}

\begin{satz}\label{pseudoadjunction}
For all $Y\in\ALGSPCFP/X$ and all $\cG\in\str{\Shv}_\et(N(X))$,
map (\ref{mapSN}) is a bijection.
\end{satz}
\begin{proof}
The restriction functors
$$ \Shv((\ALGSPCFP/X)_\et) \rightarrow \Shv((\SCHFP/X)_\et),$$
$$ \str{\Shv}((\str{\ALGSPCFP}/X)_\et) \rightarrow \str{\Shv}((\str{\SCHFP}/N(X)_\et))$$
are isomorphisms. So it is enough to see the bijection for algebraic spaces.
But there we have the commutative diagram
$$
\xymatrix{
\Hom{\str{\Shv}_\et(\str{\ALGSPCFP}/X)}{N(h_Y)}{\cG} \ar[rr]^-{\sim}_-{\text{*Yoneda}} \ar[d] & &
\Gamma(N(Y),\cG) \ar@{=}[d]\\
\Hom{\Shvet(\ALGSPCFP/X)}{h_Y}{S(\cG)} \ar[r]^-{\sim}_-{\text{Yoneda}} & \Gamma(Y,S(\cG)) \ar@{=}[r] & \Gamma(N(Y),\cG).
}
$$
\end{proof}

\vspace{\abstand}

Next we want to study the behaviour of stalks under the functor
N. For that let again $A\in\str{\cR}, X\in\SCHFP/A$, and let $\cF\in \Shvet^{fp}(X)$ be
an \'etale sheaf. If $K$ is a *artinian A-*algebra, there is by
\cite{enlsch}[Theorem 4.13] a canonical bijection
$$\Hom{\str{\SCH/A}}{\str{\spec{K}}}{N(X)}\rightarrow \Hom{\SCH/A}{\spec{K}}{X}.$$

\vspace{\abstand}

Let now $K\in\str{\cR}$ be a separably closed field and
$$\bar{x}:\spec{K}\rightarrow X$$
a geometric point of $X$.
By abuse of notation, we denote by
$$N(\bar{x}):\str{\spec{K}}\rightarrow N(X)$$
the corresponding *geometric point of $N(X)$. The stalk of $\cF$
at $\bar{x}$ is by definition
$$\cF_{\bar{x}}=\mathrm{colim}_{U}\Gamma(U,\cF)$$
where $U$ runs through the inductive system of \'etale neighbourhoods
of $\bar{x}$. If
$$\xymatrix{U \ar[r]  & X \\
      \spec{K} \ar[u] \ar[ur]_{\bar{x}}}$$
is an \'etale neighbourhood of $\bar{x}$, then
$$\xymatrix{N(U) \ar[r]  & N(X) \\
      \str{\spec{K}} \ar[u] \ar[ur]_{N(\bar{x})}}$$
is a *\'etale neighbourhood of $N(\bar{x})$, and we have the canonical
homomorphisms
$$\Gamma(U,\cF)\rightarrow \Gamma(N(U),N(\cF)).$$
These define a canonical homomorphism
\begin{equation}\label{maponstalks}
\cF_{\bar{x}}\rightarrow N(\cF)_{N(\bar{x})}.
\end{equation}
In general, this is not an isomorphism, but it is one for constructible
sheaves. For that we recall:

\vspace{\abstand}

\begin{defi}
An \'etale sheaf on a scheme $X$ is called constructible if it is representable
by an algebraic space which is finite and \'etale over $X$.
\end{defi}

\vspace{\abstand}

\begin{satz}\label{satz:stalkiso}
In the above situation, if we assume that $\cF$ is a constructible
sheaf, then the canonical morphism
$$\cF_{\bar{x}}\xrightarrow{\sim} N(\cF)_{N(\bar{x})}$$
is an isomorphism.
\end{satz}
\begin{proof}
That $\cF$ is constructible means that it is representable by an
algebraic space ${Y\rightarrow X}$, finite and \'etale over $X$. Then we
have
$$\cF_{\bar{x}}=\Hom{\ALGSPC/X}{\spec{K}}{Y}=\Hom{\ALGSPC/K}{\spec{K}}{Y\otimes_X
  K}$$
Now $Y\otimes_X K$ is a scheme, and we have again by \cite{enlsch}[Theorem 4.13] the bijection
$$\Hom{\SCH/K}{\spec{K}}{Y\otimes_X K}\xrightarrow{\sim}
\Hom{\str{\SCH}/K}{\str{\spec{K}}}{N(Y\otimes_X K)},$$
and $N(Y\otimes_X K)=N(Y)\otimes_{N(X)} K$ as well as the
identification
$$N(\cF)_{N(\bar{x})}=\Hom{\str{\SCH}/k}{\str{\spec{K}}}{N(Y)\otimes_{N(X)} K}.$$
\end{proof}

\vspace{\abstand}

\begin{bem}
We give an example showing that
map (\ref{maponstalks}) is not an
isomorphism in general. For that let $X=\spec{\str{\Q}}$,
$K=\str{\bar{\Q}}$, and let $\bar{x}$ be given by the canonical *embedding
$\str{\Q}\rightarrow \str{\bar{\Q}}$. Then $\G_{a,\bar{x}}$ is the
algebraic closure of $\str{\Q}$ in $\str{\bar{\Q}}$,
whereas $\str{\G}_{a,N(\bar{x})}$ is $\str{\bar{\Q}}$, which is surely different.
\end{bem}

\vspace{\abstand}

Next we want to remark that morphism (\ref{maponstalks}) is compatible
with specialisation morphisms. So let $K,k\in\str{\cR}$ be
separably closed fields,
$$\bar{a}: \spec{K} \rightarrow X \text{ and }
  \bar{s}: \spec{k} \rightarrow X$$
two geometric points, and
$$\varphi: \spec{\cO_{X,\bar{a}}}\rightarrow \spec{\cO_{X,\bar{s}}}$$
be a specialisation morphism, i.e. an $X$--morphism.
Let
$$ \Psi:\str{\spec{\cO_{N(X),N(\bar{s})}}}\rightarrow
\str{\spec{\cO_{N(X),N(\bar{a})}}}$$
be a *specialisation morphism that prolongs $\varphi$, i.e. the
diagram
$$
\xymatrix{
\cO_{X,\bar{s}}\ar[d]\ar[r] & \cO_{N(X),N(\bar{s})}\ar[d] \\
\cO_{X,\bar{a}}\ar[r]       & \cO_{N(X),N(\bar{a})}
}
$$
commutes.

\vspace{\abstand}
Then we have:
\begin{satz}
In the situation described above, let $\cF$ be an \'etale sheaf on $X$
which is representable by an algebraic space of finite presentation
over $X$. Then the induced diagram
$$
\xymatrix{
\cF_{\bar{s}}\ar[d]\ar[r]   & N(\cF)_{N(\bar{s})} \ar[d]  \\
\cF_{\bar{a}}\ar[r]         & N(\cF)_{N(\bar{a})}
}
$$
is commutative.
\end{satz}
\begin{proof}
This follows directly from the explicit construction of the morphisms.
\end{proof}

\vspace{\abstand}

Now we want to see how cohomology and higher derived direct images
behave under the functor $N$. First we consider the absolute case. For
that let $A\in\scR$ be an internal ring and $X\in\SCHFP/A$.

\vspace{\abstand}

\begin{lemma}\label{Sandflabby}
If $\cI\in\str{\Shvet(N(X))}$ is a *injective *sheaf on $N(X)$, then
$S(\cI)$ is a flabby sheaf on $X$.
\end{lemma}

\begin{proof}
Because all participating schemes are quasi compact, we only have to consider a
finite covering $\{U_i\rightarrow U\}$. But then $\{ N(U_i)\rightarrow
N(U)\}$ is a *covering, and we have
$$\check H^i(\{U_i\rightarrow U\},S(\cI))= \str{\check H^i(\{ N(U_i)\rightarrow
N(U)\},\cI)} =0$$
\end{proof}

\vspace{\abstand}

We consider the following commutative diagram
\begin{equation}\label{basechangediagrammabsolute}
\xymatrix{
\Shvet(X)\ar[dr]_{\Gamma(X,-)}&\str{\Shv}_\et(N(X))\ar[l]_-{S}\ar[d]^{\str{\Gamma(N(X),-)}}\\
& Ab
}
\end{equation}

\vspace{\abstand}

By Lemma \ref{Sandflabby}, we get for a $\cG\in\str{\Shv(N(X))}$ a spectral sequence
\begin{equation}
E_2^{p,q}=H_\et^p(X,\RR^q S\cG) \Rightarrow H_\et^{p+q}(N(X), \cG)
\end{equation}
and the edge homomorphism
\begin{equation}\label{absolutebasechangehilfs}
H^p_\et(X,S\cG)\rightarrow \str{H}_\et^p(N(X),\cG).
\end{equation}

\vspace{\abstand}

For a sheaf $\cF\in\ALGSPCFP/X\subset \Shvet(X)$ we compose the
natural morphism induced by (\ref{nattrans})
$$ H_\et^p(X,\cF)\rightarrow H_\et^p(X,S\circ N (\cF))$$
with (\ref{absolutebasechangehilfs}) and get
\begin{equation}\label{basechangeabsolute}
H_\et^p(X,\cF)\rightarrow \str{H}_\et^p(N(X),N(\cF))
\end{equation}

\vspace{\abstand}

\begin{bem}\label{handybasechangehom}
We would like to give an alternative and more handy description of the
above map. For given $A,X,\cF$ we find a subring $A_0\subset A$ of
finite type over $\Z$, a scheme $X_0\in\SCHFP/A_0$ and a \'etale
sheaf $\cF_0\in \Shvet(X_0)$ with isomorphisms
\begin{equation}\label{eq:2}
 X_0\otimes_{A_0} A\xrightarrow{\sim} X
\end{equation}
and
\begin{equation}\label{eq:3}
\cF\xrightarrow{\sim} \pi_{A_0}^*\cF_0,
\end{equation}
where  $\pi_{A_0}$ is the projection
$\pi_{A_0}:X_0\otimes_{A_0}A\rightarrow X_0$ and in (\ref{eq:3}) the
identification (\ref{eq:2}) is used.
We define the inductive system of subrings
\begin{equation}\label{eq:4}
L:=\{B\subset A\vert B\text{ is of finite type over } \Z \text{ and }
A_0\subset B\}.
\end{equation}
Then we have by \cite{SGA4II}[VII.5.7] a canonical isomorphism
\begin{equation}
H^i_\et(X,\cF)\xrightarrow{\sim} \mathrm{colim}_{B\in L}
H^i_\et(X_0\otimes_{A_0}B, \pi_B^*\cF_0),
\end{equation}
where $\pi_B$ denotes the projection $\pi_B:X_0\otimes_{A_0}
B\rightarrow X_0$.
Now for each $B\in L$ there is a canonical morphism
\begin{equation}
H^i_\et(X_0\otimes_{A_0}B,\pi^*_B\cF_0)\rightarrow
  \str{H}^i_\et(\str{(X_0\otimes_{A_0}B)},\str{(\pi^*_B\cF_0)})
  \rightarrow
    \str{H}^i_\et(N(X),N(\cF)),
\end{equation}
because $N(X)\cong \str{(X_0\otimes_{A_0}B)}\otimes_{\str{B}}A$.
These morphisms induce a morphism
\begin{equation}
H^i_\et(X,\cF)\rightarrow \str{H}^i_\et(N(X),N(\cF)),
\end{equation}
which is identical to the morphism we have just constructed above.
\end{bem}
\vspace{0.5cm}

\vspace{\abstand}
For the relative case we consider an internal ring $A\in\scR$, two schemes
$X,Y\in\SCHFP/A$ and a morphism of $A$-schemes $f:X\rightarrow Y$.
\vspace{\abstand}

We will construct a ``base change'' homomorphism in analogy to the
usual base change homomorphism in \'etale cohomology. We consider the
commutative diagram
\begin{equation}\label{basechangediagramm}
\xymatrix{
\Shvet(X)\ar[d]^{f_*}&\str{\Shv}_\et(N(X))   \ar[l]^-{S}\ar[d]^{N(f)_*}\\
\Shvet(Y)            &{\str{\Shv}_\et(N(Y)).}\ar[l]^-{S}
}
\end{equation}

\vspace{\abstand}

For a *sheaf $\cG\in\str{\Shvet}(N(X))$ we have by Lemma \ref{Sandflabby} the spectral sequence
$$ E_2^{p,q}=\RR^pf_*\RR^qS \cG \Rightarrow \RR^{p+q} (f_*\circ S) \cG$$
and the edge homomorphisms
\begin{equation}\label{edgehom1}
\RR^pf_*[S\cG]\rightarrow \RR^p(f_*\circ S) \cG.
\end{equation}
Because $N(f)_*$ maps *injectives to *injectives, we further have the
spectral sequence
$$ E_2^{p,q}=\RR^pS\RR^qN(f)_* \cG\Rightarrow \RR^{p+q}(S\circ N(f)_*)
\cG$$
with edge homomorphisms
\begin{equation}\label{edgehom2}
\RR^q (S\circ N(f)_*)\rightarrow S\RR^qN(f)_* \cG.
\end{equation}
By the commutativity of (\ref{basechangediagramm}) we can compose
(\ref{edgehom1}) and (\ref{edgehom2}) to get a morphism
\begin{equation}\label{prebasechangehom}
\RR^q f_*[S\cG]\rightarrow S\RR^q N(f)_* \cG.
\end{equation}

\vspace{\abstand}

For an $\cF\in\Shvet^{fp}(X)$ we get
\begin{equation}
\RR^qf_*[SN\cF] \rightarrow S \RR^q N(f)_*[N\cF]
\end{equation}
and then with (\ref{nattrans})
\begin{equation}\label{bbaschange}
\RR^qf_* \cF \rightarrow S \RR^q N(f)_*[N\cF].
\end{equation}

\vspace{\abstand}

\begin{bem}
As in \ref{handybasechangehom}, we can also give an easier description of
this map. The construction of \ref{handybasechangehom} gives a map of
presheaves on $\SCHFP/Y$ from the presheaf
$$(U\rightarrow Y)\mapsto H^i_\et(X\times_Y U,\cF)$$
to the presheaf
$$(U\rightarrow Y)\mapsto \str{H}^i_\et(N(X\times_Y U),N(\cF)).$$
Sheafification then induces morphism (\ref{bbaschange}).
\end{bem}

\vspace{\abstand}

If we assume further that the sheaves $\RR^qf_* \cF$
are constructible for all $q\geq 0$, we get by \ref{pseudoadjunction}
the base change homomorphism
\begin{equation}\label{basechangehom}
N\RR^qf_* \cF \rightarrow \RR^q N(f)_*[N\cF].
\end{equation}

\noindent
For proper morphisms, we know that this base change homomorphism is actually an
isomorphism:
\vspace{\abstand}
\begin{satz}\label{satz:etalecentral}
Let $f:X\rightarrow Y$ be a proper morphism with $X,Y\in\SCHFP/A$, and
$\cF$ be a constructible sheaf on $X$. Then the base change homomorphism
\begin{equation}\label{morphism:basechangeiso}
N\RR^qf_* \cF \rightarrow \RR^q N(f)_*[N\cF]
\end{equation}
is an isomorphism.
\end{satz}
\begin{proof}
By \cite{ega43}[8.8.2] and lemma \ref{fundlemmaN}, we can choose $A_0\subset A$ of finite type
over $\Z$, schemes $X_0, Y_0\in \SCHFP/A_0$ with
$X_0\otimes_{A_0}A\xrightarrow{\sim} X$ and
$Y_0\otimes_{A_0}A\xrightarrow{\sim} Y$,
$\cF_0\in\ALGSPCFP/X_0\subset\Shvet/X_0$ with
$\pi_X\cF_0\xrightarrow{\sim}\cF$
and a proper(!) morphism
$f_0:X_0\rightarrow Y_0$, such that the diagram
$$
\xymatrix@C=20mm{
X_0\otimes_{A_0}A\ar[r]^{f_0\otimes id_A} \ar[d]^{\sim} &
Y_0\otimes_{A_0} A \ar[d]^{\sim}\\
X \ar[r]_{f} & Y
}
$$
commutes. By the base change theorem for proper morphisms for the
cartesian square
$$
\xymatrix{
X\ar[r]^{\pi_X} \ar[d]_{f} \ar@{}[dr]|{\Box} & X_0 \ar[d]^{f_0} \\
Y\ar[r]_{\pi_Y}                              & Y_0
}
$$
we have
\begin{equation}\label{proofbasechangefirst}
\RR^if_* \cF\simeq \RR^if_*\pi_X^*\cF_0 \simeq \pi_Y^* \RR^if_{0*} \cF_0.
\end{equation}
By construction we have
\begin{equation}
N(X)\simeq \str{X}_0\otimes_{\str{A}_0}A,
N(Y)\simeq \str{Y}_0\otimes_{\str{A}_0}A,\text{ and }
N(f)=\str{f_0}\otimes_{\str{A}_0}A.
\end{equation}
 So we have the cartesian square
\begin{equation}\label{proofbasechangeN}
\xymatrix@C=12mm@R=12mm{
N(X) \ar[r]^{\pi_{N(X)}} \ar[d]_{N(f)} \ar@{}[dr]|{\Box} &
\str{X}_0\ar[d]^{\str{f}_0}\\
N(Y)\ar[r]_{\pi_{N(Y)}}    &  \str{Y}_0
}
\end{equation}
and also by construction the identification
$$ N(\cF)\simeq \pi_{N(X)}^*\str{\cF_0}.$$
By the *base change theorem for the *proper morphism $\str{f}_0$ in diagram
(\ref{proofbasechangeN}), we have
\begin{equation}\label{proofbasechangesecond}
\RR^i N(f)_* N(\cF) \simeq \RR^i N(f)_* \pi_{N(X)}^*\str{\cF_0} \simeq \pi_{N(Y)}^* \RR^i(\str{f}_0)_* \str{\cF}_0.
\end{equation}
Because * is exact and maps injectives to injectives (cf
\cite{nsetale}), we have
\begin{equation}\label{changehilfeeq}
\str{(\RR^i
  f_{0*}\cF_0)})\simeq \RR^i(\str{f}_0)_*\str{\cF}_0.
\end{equation}
Now we get what we want:
$$ N(\RR^i f_* \cF)\stackrel{~(\ref{proofbasechangefirst})}{\simeq} \pi_{N(Y)}^*(\str{(\RR^i
  f_{0*}\cF_0)})\stackrel{~(\ref{changehilfeeq})}{\simeq} \pi_{N(Y)}^*(\RR^i(\str{f}_0)_*
\str{\cF}_0)\stackrel{(\ref{proofbasechangesecond})}{\simeq}
\RR^i N(f)_* N(\cF).$$
\end{proof}

\vspace{\abstand}

\begin{cor}\label{cor:basechangeabsolute}
Let $K\in\str{\cR}$ be a separably closed field, $f:X\rightarrow \spec{K}$
proper and $\cF$ a constructible \'etale sheaf on $X$. Then the canonical
morphism \eqref{basechangeabsolute}
$$H^i_\et(X,\cF)\rightarrow \str{H}^i_\et(N(X),N(\cF))$$
is an isomorphism.
\end{cor}

\begin{proof}
We just take the section of (\ref{morphism:basechangeiso}) on
$\str{\spec K}$ and identify
$$ \Gamma(\str{\spec K},N(\RR^i f_*\cF))\cong \Gamma(\spec K, \RR^i
f_*\cF)$$
by \cite{enlsch}[Theorem 4.13] as in the proof of Proposition
\ref{satz:stalkiso}.
\end{proof}

Now we want to prove a compatibility between the just defined morphism on
cohomology and a morphism on the Picard group defined in \cite{enlsch}:

\vspace{\abstand}

\begin{satz}\label{h1undpicundN}
Let $A\in\scR$ and $X\in\SCHFP/A$. Then the diagram
\begin{equation}
\xymatrix{
H^1_\et(X,\G_m)\ar[r]^-{(i)} \ar[d]_{\wr} &
   \str{H}^1_\et(N(X),\G_m)\ar[d]^{\wr} \\
\Pic{X}\ar[r]_-{(ii)} &\sPic{N(X)}
}
\end{equation}
is commutative. Here (i) is defined by (\ref{basechangeabsolute}),
and (ii) is defined by \cite{enlsch}[Corollary 5.15].
\end{satz}
\begin{proof}
Both horizontal maps are defined using a model $X_0$ of $X$ which is defined
over a subring $A_0\subset A$ of finite presentation over
$\Z$. Therefore we only have to show that the diagram
$$
\xymatrix{
H^1_\et(X_0,\G_m)\ar[r]^-{(i)} \ar[d]_{\wr} &
\str{H}^1_\et(\str{X_0},\str{\G_m})\ar[d]^{\wr} \\
\Pic{X_0}\ar[r]_-{(ii)} &\sPic{\str{X_0}}
}
$$
is commutative. But this is clear.
\end{proof}

\vspace{\abstand}

\section{Independence of $l$ of Betti numbers}

Now we give an application of the first section
to the problem of the independence of $l$ of Betti numbers for the
\'etale cohomology of separated schemes of finite type over finite
fields. The following is conjectured:
\begin{conj}\label{sec:an-applicationconjecture}
Let $k$ be a field finite field, $\bar{k}$ be an algebraic closure
of $k$, and $X$ a separated scheme of finite type over $k$. Then the
dimension of the $l$-adic cohomology with compact support
$$\dim_{\Q_l}H^i_{c,\et}(X\otimes_k\bar{k},\Q_l)$$
is independent of $l$. (see for example \cite{katz})
\end{conj}

It is well known that the corresponding statement is true when the ground field
$k$ is of characteristic zero (cf. \cite{SGA4III}[exp. XVI, 4]).
Furthermore, it is generally believed that a
theorem which is true for fields of characteristic zero is also true
for fields of large positive characteristic. The aim of this section
is to turn this belief into a precise statement in the case of the
independence of $l$ of Betti numbers.

First we prove the following general result about the dimension of
$l$-adic cohomology.

\vspace{\abstand}

\begin{thm}\label{ladicdimensionandN}
Let $B\in\cR$ be of finite type over $\Z$, $X\xrightarrow{f}\spec{B}$ a
proper morphism and $(\cG_n)_{n\in\N}$ an AR-$l$-adic system of
constructible \'etale sheaves on $X$. Let $K\in\str{\cR}$ be an algebraically
closed field and $\str{B}\rightarrow K$ an internal
homomorphism. Then we have
$$\dim_{\Q_l}(\lim_{n\in\N} H^i_\et(X_K,\cG_n)\otimes_{\Z_l}\Q_l) =
\dim_{\str{\Q_l}}(\str{\lim}_{n\in\str{\N}}
\str{H}^i_\et(\str{X}_K,\str{\cG}_n)\otimes_{\str{\Z_l}}\str{\Q_l})
$$
\end{thm}
\begin{proof}
We have $N(X_K)=\str{X}_K$ and $N(\cG_n)=\str{\cG_n}$ on
$\str{X_K}$. Therefore by (\ref{cor:basechangeabsolute}) the canonical morphism
$$ H^i_\et(X_K,\cG_n)\rightarrow
\str{H}^i_\et(\str{X}_K,\str{\cG}_n)$$
is an isomorphism for all $n\in\N$. If both sides were $l$-adic respectively $\str{l}$-adic
systems, the claim would follow. But by applying the next proposition to the AR-$l$-adic
system $(\RR^if_*\cG_n)_{n\in\N}$ and using the fact that
$$    \str{\RR}^i f_*\cG_n\xrightarrow{\sim}\RR^i(\str{f})\str{\cG_n}$$
is an isomorphism,
we see that we only need a finite number of terms
to calculate a term of $l$-adic respectively $\str{l}$-adic systems which are AR-isomorphic to the systems above.
\end{proof}

\vspace{\abstand}

\begin{satz}
Let $X$ be a noetherian scheme and $\cG=(\cG_n)_{n\in\N}$ be an
AR-$l$-adic system of constructible \'etale sheaves on $X$. Then there
are constants $n_0,n_1,n_2\in\N$ with the following property: If we
define
$$\cH_n:=\cG_n/\ker(\cG_n\xrightarrow{\cdot l^{n_0}}\cG_n)
\text{ and } \cF_n:=\im(\cH_{n_1+n_2+n}\rightarrow
\cH_{n_2+n})/l^{n+1},$$
the system $(\cF_n)_{n\in\N}$ is a torsion free $l$-adic system which
is up to torsion AR-$l$-isomorphic to $\cG$.
\end{satz}
\begin{proof}
The category of AR-$l$-adic systems of constructible \'etale sheaves is
by \cite{kiehl1}[Prop. 12.12] noetherian. Therefore there is an $n_0$ so that for all $m>n_0$
the inclusion
$$\ker(\cG\xrightarrow{\cdot l^{n_0}} \cG)\rightarrow
\ker(\cG\xrightarrow{\cdot l^{m}} \cG)$$
is an AR-isomorphism. So
$$\cG/\ker(\cG\xrightarrow{\cdot l^{n_0}}\cG)=:\cH$$
is an AR-torsion free AR-$l$-adic system. Now for each AR-$l$-adic
system $\cH$ there are integers $n_1,n_2\in\N$ such that
$$\cF_n:=\im(\cH_{n_1+n_2+n}\rightarrow \cH_{n_2+n})/l^{n+1}$$
is an $l$-adic sheaf which is AR-isomorphic to $\cH$.
\end{proof}

\vspace{\abstand}

Now we restrict ourselves to projective varieties to have an easier
notion of complexity. For natural
numbers $n,d\in\N$ and a field $k$ we define $H(n,d,k)$ as the set of
all closed subschemes of $\P^n_k$ of degree $d$.

\vspace{\abstand}
Now we consider the function $B_i:=B^{d,n}_i:\PP\rightarrow
\N\cup\{\infty\}$ on the set of prime numbers $\PP$ which is defined by
$$p\mapsto\max\left\{m\in\N\left\vert
\begin{array}{c}
\text{ for all finite fields } k \text{ with }\ch{k}=p\\
\text{ and all } X\in H(n,d,k) \text{ and all primes } l_1,l_2<m\\
\text{ we have }\dim_{\Q_{l_1}}H^i_\et(\bar{X},\Q_{l_1}) =
\dim_{\Q_{l_2}}H^i_\et(\bar{X},\Q_{l_2})
\end{array}\right.\right\}
$$

\vspace{\abstand}

Note that conjecture \ref{sec:an-applicationconjecture} says $B^{d,n}_i\equiv\infty$.
From the fact that the Conjecture holds in characteristic zero, we can deduce the following with our methods:

\vspace{\abstand}

\begin{thm}\label{indepofl}
With the above notations we have:
$$\lim_{p\rightarrow\infty} B_i^{d,n}(p)=\infty$$
\end{thm}
\begin{proof}
Let us assume the statement is not true. Then by transfer there are an
infinite prime $P\in\str{\PP}\setminus\PP$, a *finite field $k$ of internal
characteristic $P$, a *scheme $X\hookrightarrow\str{\P}^n_k$ of *degree $d$
and two standard primes $l_1,l_2\in\PP$ such that
$$\dim_{\str{\Q}_{l_1}}\str{H}^i_\et(\bar{X},\str{\Q}_{l_1})\ne
\dim_{\Q_{l_2}}\str{H}^i_\et(\bar{X},\str{\Q}_{l_2}).$$
But $X$ is of *degree $d\in\N$ in $\str{\P}^n_k$, so that by
\cite{enlsch}[6.21], there is a $X'\in\SCHFP/\str{\bar{k}}$ with $N(X')=\bar{X}$.
Then by \ref{ladicdimensionandN} we have for all standard primes
$l\in\PP$
$$\dim_{\Q_l}H^i_\et(X',\Q_l)=\dim_{\str{\Q}_l}\str{H}^i_\et(\bar{X},\str{\Q}_l).$$
Therefore we have a contradiction to the independence of $l$ for fields
of characteristic zero.
\end{proof}

\vspace{\abstand}

\section{Voevodsky motives and cycles}

The aim of this section is to construct a functor $N$ for the motivic
cohomology of schemes. For that, we use the geometric construction of
the triangulated category of mixed motives by Voevodsky from
\cite{Voetriangl}. The advantage of this way is that we only have to deal with
finite correspondences and proper intersections.

For the convenience of the reader, we
shortly recall the construction of Voevodsky's triangulated category of
geometrical motives. For details we refer to \cite{Voetriangl}
and \cite{MaVoWe}.
After that, we discuss enlargements of these motives and
--- most importantly --- the existence of a functor $N$ for them.

\vspace{\abstand}

For a field $k$, we denote by $\Sm/k$ the category of smooth schemes of
finite type over $k$. For $X\in \SCHFP/k$, we denote by
$Z(X)$ the group of algebraic cycles.
Let $V\subseteq X$ be a closed subscheme.
Recall that $[V]$, the \emph{cycle associated to $V$},
is defined as
\[
  \sum_{x}\length{\cO_{X,x}}{\cO_{V,x}}\in Z(X),
\]
where the sum is taken over the generic points $x$ of the irreducible
components of $V$ and $\length{\cO_{X,x}}(\cdot)$ denotes length of an
$\cO_{X,x}$-module.

\vspace{\abstand}

For two schemes $X,Y\in \SCHFP/k$, we denote by $c(X,Y)$
the subgroup of $Z(X\times Y)$ generated by cycles of $X\times Y$ which are
finite and surjective over an irreducible component of $X$ via the
projection $X\times Y\rightarrow X$. Elements of $c(X,Y)$ are called
\emph{finite correspondences}.

\vspace{\abstand}

Now let $X,Y,Z\in \Sm/k$, $W_1\hookrightarrow X\times Y$ an irreducible
closed subscheme, finite and surjective over a component of $X$, and
$W_2\hookrightarrow Y\times Z$ an irreducible
closed subscheme, finite and surjective over a component of $Y$. One
point of using finite correspondences is that $W_1\times Z$ and
$X\times W_2$ intersect properly on $X\times Y \times Z$. So we can
define the intersection product $[W_1\times Z].[X\times W_2]$, and it is
a sum of prime cycles which are finite over $X$. Therefore one can
define $W_1\circ W_2:= p_{13*}([W_1\times Z].[X\times W_2])\in c(X,Z).$ In
particular, we do not have to work with rational equivalence, and we do
not have to use a moving lemma. This extends to a composition of
finite correspondences. With this composition we get an additive
category $SmCor(k)$ of smooth correspondences where the objects are
smooth schemes of finite type over $k$ and where the morphisms are
finite correspondences.

\vspace{\abstand}

The graph $\Gamma_f$ of a usual morphism $f:X\rightarrow Y$ gives us a
covariant functor $[-]:\Sm/k\rightarrow SmCor(k)$.
\vspace{\abstand}

Now we consider the homotopy category $\mathcal{K}^{b}(SmCor(k))$ of
bounded complexes in $SmCor(k)$, and we let $T$ be the smallest thick
subcategory of the triangulated category $\mathcal{K}^{b}(SmCor(k))$
which contains the following types of complexes:
\vspace{\abstand}
\begin{enumerate}
\item $[X\times\mathbb{A}^1]\xrightarrow{[pr_{1}]} [X]$ for all $X\in
  \Sm/k$
\item $[U\cap V]\xrightarrow{} [U]\oplus [V] \xrightarrow{} [X]$ for all
  \\ $X\in \Sm/k$ and Zariski open coverings $X=U\cup V$ of $X$ with the
  obvious morphisms.
\end{enumerate}
\vspace{\abstand}

Then let $DM^{eff}_{gm}(k)$ be the pseudo abelian hull of
$\mathcal{K}^{b}(SmCor(k))/T$.

It turns out that $DM^{eff}_{gm}(k)$ is again a triangulated category,
that the product $[X]\otimes [Y]=[X\times Y]$ for $X,Y\in \Sm/k$
defines a tensor triangulated structure on $DM^{eff}_{gm}(k)$ and that
we have a canonical functor
$$M_{gm}:\Sm/k\rightarrow DM^{eff}_{gm}(k).$$

We denote by $\Z$ the object $M_{gm}(\spec{K}$. It is given by
the complex
$$\cdots \rightarrow 0 \rightarrow [\spec{K}]\rightarrow 0 \rightarrow
0\cdots.$$

\vspace{\abstand}

We denote by $\Z/n$ the object in $DM_{gm}^{eff}$ which is
given by the complex
$$\cdots\rightarrow 0\rightarrow [\spec{K}] \xrightarrow{\cdot n}
[\spec{K} ]\rightarrow 0 \rightarrow \cdots$$
living in degree $-1$ and $0$. So we have the exact triangle
$$\Z\xrightarrow{\cdot n}\Z \rightarrow \Z/n
\rightarrow \Z[1]$$
in $DM_{gm}^{eff}$.

\vspace{\abstand}

The Tate object $\Z(1)\in DM^{eff}_{gm}(k)$ is defined to be the
image of the complex $[\mathbb{P}^1]\rightarrow [\spec{k}]$, where
$[\mathbb{P}^1]$ sits in degree 2. For an $n\in\mathbb{N}$ we define
$\Z(n):=\Z(1)^{\otimes n}$, and for an object $A\in
DM^{eff}_{gm}(k)$ we set $A(n):=A\otimes \Z(n)$.

Finally we get $DM_{gm}(k)$ by inverting $\Z(1)$, and it can be
shown that the tensor structure lifts from $DM^{eff}_{gm}(k)$ to
$DM_{gm}(k)$.

\vspace{\abstand}

For varying fields $k\in\cR$ we get for each construction step a fibration of
categories over the categories of fields in $\cR$ and if we choose an
appropriate superstructure we get for each internal field $K\in\str{\cR}$ internal
categories
$\str{\Sm/K}, \str{SmCor(K)}$,
$\str{\mathcal{K}^{b}(SmCor(K))/T}$,
$\str{DM^{eff}_{gm}(K)}$ and
$\str{DM_{gm}(K)}$.
For the tensor product in $\str{DM_{gm}(K)}$ we
write again $\otimes$ instead of $\str{\otimes}$, and we again have the Tate object
$\str{\Z}(1)\in\str{DM_{gm}(K)}$. For an
$n\in\str{\Z}$ we define $\str{\Z}(n):=\str{\Z}(1)^{\otimes n}$,
and for an object $A\in \str{DM_{gm}(K)}$ we set $A(n):=A\otimes \str{\Z}(n)$.
For each standard field $k$ we get a functor of $\otimes$-triangulated
categories
$$*:DM_{gm}(k)\rightarrow \str{DM_{gm}(\str{k})}.$$
with $\str{(\Z(1))}=\str{\Z}(1)$.

\vspace{\abstand}

For an internal field $K\in\cR$ we want to define a functor
$$ N:DM_{gm}(K)\rightarrow \str{DM_{gm}(K)}.$$

First we have functors
$$N:\Sm/K\rightarrow \str{\Sm}/K$$
and
$$N:\SCHFP/K\rightarrow \str{\SCHFP}/K$$
which are constructed and analysed in \cite{enlsch}[section 4].

\vspace{\abstand}

Since $N:\SCHFP/K\longrightarrow\str{\SCHFP}/K$ maps
prime cycles to *prime *cycles by \cite[6.4]{enlsch}, we get an induced
group homomorphism
\[
  \s{}:Z(X)\longrightarrow\str{Z}(\s{X}),\;\;\;
  \sum_{j=1}^n\alpha_j\cdot Z_j\mapsto\sum_{j=1}^n\alpha_j\cdot\s{Z_j}
\]
for each $X\in\SCHFP/K$.

\begin{satz}\label{sa:NcomWassCycles}
  $N$ commutes with taking the associated cycle,
  i.e.
  \[
    \s{[V]}=\str{[\s{V}]}\in\str{Z}^i(\s{X})
  \]
  for every closed subscheme $V$ of $X$ of codimension $i$.
\end{satz}

\begin{proof}
  Let $V$ be a closed subscheme of $X$ of codimension $i$.
  Without loss of generality, we can assume that $V$ is irreducible
  with generic point $x$.
  Furthermore, we can assume that $X=\spec{A}$ is affine,
  so that $V$ corresponds to an ideal $\a$ of $A$
  and $x$ corresponds to a prime ideal $\p$ which is the unique minimal
  prime ideal above $\a$.
  Then \ref{lemmaLength}, applied to $M:=A/\a$,
  shows
  \[
    \length{\cO_{X,x}}{\cO_{V,x}}=
    \length{A_\p}{A_\p/\a A_\p}=
    \slength{[\s{A}]_{\s{\p}}}{[\s{A}]_{\s{\p}}/\a[\s{A}]_{\s{\p}}}=
    \slength{\cO_{\s{X},x'}}{\cO_{\s{V}},x'},
  \]
  where $x'$, the point given by the *prime ideal $\s{\p}=\p[\s{A}]$,
  is the generic point of $\s{V}$ by \cite[2.7]{vandendries}.
  By definition of "associated cycle" this finishes the proof.
\end{proof}

\vspace{\abstand}

\begin{satz}\label{sa:NcomWpf}
$N$ commutes with push forward of cycles along proper morphisms.
\end{satz}
\begin{proof}
Let $f:X\rightarrow Y$ be a proper morphism with $X,Y\in \SCHFP/K$, and
let $W\subset X$ be a closed integral subscheme. By definition we have
$$f_*[W]=\left\{ \begin{array}{cl} [\kappa(W):\kappa(f(W))]\cdot [f(W)] &
  \text{if } \dim(W)=\dim(f(W)) \\
  0 & \text{otherwise,} \end{array}\right.$$
and the proposition follows from \cite{enlsch}[Lemma 6.27] and
\cite{enlsch}[Theorem 6.4].
\end{proof}

\vspace{\abstand}

Now for $X,Y\in \Sm/K$, let  $W\hookrightarrow X\times Y$
be a cycle which is finite and surjective over a connected component of
$X$. Then again by \cite{enlsch}[section 4],
$$N(W)\hookrightarrow N(X\times Y)=N(X)\times N(Y)$$
is a cycle, *finite and surjective over $N(X)$, i.e. we get a morphism
$$c(X,Y)\rightarrow \str{c}(N(X),N(Y)).$$
With that we get the following theorem:
\begin{satz}\label{sa:NforSmCor}
The above construction defines a functor
$$N:SmCor(K)\rightarrow\str{SmCor}(K).$$
\end{satz}
\begin{proof}
We have to show that the construction is compatible with composition
in the categories. By Proposition \ref{sa:NcomWassCycles} and
\ref{sa:NcomWpf} it is enough to show that $N$ commutes
with the intersection product of two cycles which intersect properly
on a smooth scheme. By reduction to the diagonal, the occurring multiplicities
are multiplicities of a Koszul complex. Because $N$ is exact on
modules (cf \cite{enlsch}[Theorem 6.4]), it is enough to show
the compatibility of $N$ with the Koszul complex and with the notion of
length. But this is done in lemma \ref{lemmaLength} and lemma \ref{le:NcomWithKoszul}.
\end{proof}

\vspace{\abstand}

The functor $N:SmCor(K)\rightarrow\str{SmCor}(K)$ induces a functor
$\mathcal{K}^b(SmCor(K)) \rightarrow \mathcal{K}^b(\str{SmCor})(K)$. We
compose this with the canonical functor
$$\mathcal{K}^b(\str{SmCor}(K))\rightarrow
\str{\mathcal{K}}^b(\str{SmCor}(K)$$
which was studied in \cite{enlcat}[Section 6] to get the functor
$$N:\mathcal{K}^b(SmCor(K)) \rightarrow
\str{\mathcal{K}}^b(\str{SmCor})(K).$$
\vspace{\abstand}

Again by \cite{enlsch}[Section 4] we have
\begin{itemize}
\item $N([X\times\mathbb{A}^1]\xrightarrow{[pr_{1}]} [X])=
[N(X)\times\str{\mathbb{A}}^1]\xrightarrow{[pr_{1}]} [N(X)]$ for all $X\in
  \Sm/K$
\item $N([U\cap V])\xrightarrow{} [U]\oplus [V] \xrightarrow{}
  [X])= [N(U)\cap N(V)]\xrightarrow{} [N(U)]\oplus [N(V)] \xrightarrow{} [N(X)]$ for all $X\in
  \Sm/k$,
\end{itemize}
and if $U\cup V=X$ is an open covering of $X$, then $N(U)\cup
N(V)=N(X)$ is an open covering of $N(X)$. Hence we have $N(\mathcal{T})\subset \str{\mathcal{T}}$. Therefore we get a
  functor
$$N:\mathcal{K}^b(SmCor(K))/\mathcal{T} \rightarrow
\str{\mathcal{K}}^b(\str{SmCor})(K)/\str{\mathcal{T}}.$$
\vspace{\abstand}
By the universal property of the pseudo abelian hull this further
induces a functor
$$N:DM_{gm}^{eff}(K)\rightarrow \str{DM}_{gm}^{eff}(K)$$
and then again by a universal property a functor
$$N:DM_{gm}(K)\rightarrow \str{DM}_{gm}(K).$$
\vspace{\abstand}
Furthermore, by \cite{enlsch}[Section 4] we have
$$N(X\times Y)=N(X)\times N(Y)$$
and
$$N([\mathbb{P}^1_K]\rightarrow
[\spec{K}])=[\str{\mathbb{P}}^1_K]\rightarrow [\str{\spec{K}}],$$
and therefore $N:DM_{gm}(K)\rightarrow \str{DM}_{gm}(K)$ is compatible
with the tensor structure on both sides, and we have
$N(\Z(n))=\str{\Z}(n)$.
\vspace{\abstand}
To summarise this we formulate the next
\begin{satz}\label{sa:NforDM}
Let $K$ be an internal field. The functor $N: SmCor(K)\rightarrow
\str{SmCor(K)}$ of Proposition \ref{sa:NforSmCor} induces a natural functor of
tensor triangulated categories
$$N:DM_{gm}(K)\rightarrow \str{DM}_{gm}(K),$$
and the diagram
$$
\xymatrix{
\Sm/K \ar[r]^{N}\ar[d]^{M} & \str{\Sm}/K \ar[d]^{\str{M}} \\
DM_{gm}(K) \ar[r]^{N}     & \str{DM}_{gm}(K)
}
$$
is commutative.
\end{satz}
\vspace{\abstand}

Now we want to show how this functor $N$ induces a morphism for
\emph{motivic cohomology}. First we recall the definition of motivic cohomology in terms of
Voevodsky's triangulated category of motives.

\vspace{\abstand}

Let $X$ be a smooth scheme of finite type over a field $k$.

\begin{defi}[motivic cohomology]
The
motivic cohomology of $X$ is defined as
$$H^i_{\mathcal{M}}(X,\Z(j)):=\Hom{DM_{gm}(k)}{M_{gm}(X)}{\Z(j)[i]}.$$
The $\otimes$-triangulated structure and the pullback along the
diagonal define on $\oplus_{i,j\in\mathbb{N}}
H^i_{\mathcal{M}}(X,\Z(j))$ a graded ring structure.
In the same way we define motivic cohomology with finite coefficients
$$H^i_{\mathcal{M}}(X,\Z/n(j)):=\Hom{DM_{gm}(k)}{M_{gm}(X)}{\Z/n(j)[i]}.$$
\end{defi}
\vspace{\abstand}

For an internal field $K$ and an $X\in\str{\Sm}/K$ we define
\begin{defi}
The *motivic cohomology of $X$ is defined as
$$\str{H}^i_{\mathcal{M}}(X,\str{\Z}(j))
:=\Hom{\str{DM}_{gm}(k)}{\str{M}_{gm}(X)}{\str{\Z}(j)[i]} $$
resp.
$$\str{H}^i_{\mathcal{M}}(X,\str{\Z/n}(j))
:=\Hom{\str{DM}_{gm}(k)}{\str{M}_{gm}(X)}{\str{\Z/n}(j)[i]} $$

for $i,j,n\in\Ns$. In the same way as above
$$\str{\bigoplus_{i,j\in\Ns}}
\str{H}^i_{\mathcal{M}}(X,\str{\Z}(j))$$
is a *graded ring
\end{defi}

\vspace{\abstand}

\begin{bem}
An alternative way to define *motivic cohomology is the following: For
each field $k$, motivic cohomology is a functor
$$H^{\cdot}(-,\Z(\cdot)):(\Sm/k)^{op}\rightarrow \text{graded
  rings}.$$
By transfer, for details see \cite{enlcat}, we get for each internal
field $K$ an (internal) functor
$$\str{H}^{\cdot}(-,\Z(\cdot)):(\str{\Sm}/K)^{op}\rightarrow
\text{(internal) graded
  rings}.$$
Obviously this agrees with the above definition.
\end{bem}

\vspace{\abstand}

\begin{bem}\label{rem:motcohandhighchow}
For a smooth equidimensional scheme $X$ we have by \cite{Voemocohchgr}
the following identification
$$ H^i_{\mathcal{M}}(X,\Z(j))=CH^j(X,2j-i),$$
where $CH^*(X,*)$ denotes the higher Chow groups of Bloch. In
particular we have
$$H^{2i}_{\mathcal{M}}(X,\Z(i))=CH^i(X),$$
where $CH^i(X)$ denotes the usual Chow groups, i.e. cycles of codimension $i$ on $X$
modulo rational equivalence.
\end{bem}

\vspace{\abstand}

Now let $K$ be an internal field.

\begin{satzdefi}\label{de:NfMoCoh}
The functor $N:DM_{gm}\rightarrow \str{DM}_{gm}(K)$ of Proposition
\ref{sa:NforDM} induces for all $i,j\in\mathbb{N}_0$ and all $X\in
\Sm/K$ a natural morphism
$$N:H^i_{\mathcal{M}}(X,\Z(j))\rightarrow
\str{H}^i_{\mathcal{M}}(N(X), \str{\Z}(j))$$

and then a morphism of graded rings
$$N:\bigoplus_{i,j\in\mathbb{N}_0} H^i_{\mathcal{M}}(X,\Z(j))\rightarrow
\str{\bigoplus}_{i,j\in\str{\mathbb{N}_0}}\str{H}^i_{\mathcal{M}}(N(X), \str{\Z}(j)).
$$
\end{satzdefi}

\vspace{\abstand}

\begin{bem}
With remark \ref{rem:motcohandhighchow} we see that we get a morphism
$$N:CH^i(X)\rightarrow \str{CH}^i(N(X))$$,
and by the construction it is easy to see that for a prime cycle
$[Y]$, we have
$$N([Y])=[N(Y)].$$
\end{bem}

\vspace{\abstand}

For a smooth scheme $X$ over a field $K$ and an $n\in\N$, prime to
$\ch{K}$, there is the cycle class map
$$cl_n: CH^i(X)\rightarrow H^{2i}_\et(X,\mu_n^{\otimes i}).$$

By transfer, for a *smooth schemes $X$ over an internal field $K$ and
an $n\in\str{\N}$, *prime to $\sch{K}$. we
have the induced map
$$\str{cl}_n: \str{CH}^i(X)\rightarrow
\str{H}^{2i}_\et(X,\str{\mu}_n^{\otimes i}).$$
These are compatible with $N$:

\vspace{\abstand}

\begin{satz}
Let $K$ be an internal field, $X$ be a smooth scheme of finite type
over $K$, and $n\in\N$ prime to $\ch{K}$. Then the diagram
$$
\xymatrix@C=20mm{
  CH^i(X) \ar[d]_{cl} \ar[r]^{N} & \str{CH}^i(N(X))\ar[d]^{\str{cl}}\\
  H^{2i}_\et(X,\mu_n^{\otimes i}) \ar[r]_{N} &
                               \str{H}^{2i}(N(X),\str{\mu}_n^{\otimes
                                 i})
}
$$
is commutative.
\end{satz}
\begin{proof}
For $i=1$ the map $cl_n$ is defined as follows. One first identifies
$$CH^1(X)\xrightarrow{\sim} Pic(X)\xrightarrow{\sim} H^1_\et(X,\G_m)$$
and then uses the connection homomorphism
$$H^1_\et(X,\G_m)\xrightarrow{\delta} H^2_\et(X,\mu_n)$$
of the short exact sequence
$$1\rightarrow \mu_n\rightarrow
\G_m\xrightarrow{\cdot^n}\G_m\rightarrow 1.$$
The same is true for $\str{cl}_n$ and therefore in this case the claim
follows from Proposition \ref{h1undpicundN}. By the compatibility of $N$ with the intersection
product \ref{sa:NforDM} the diagram is commutative for cycles which are
products of divisors. Then the cohomological methods of
\cite{SGA4.5}[Cycles,sect. 2.2] reduce the general case to this case.
\end{proof}
\vspace{\abstand}
\section{Complexity of cycles}

In this section we give a notion of complexity of a cycle and show how
that can be used to describe the image of $N:CH^i(X)\rightarrow
\str{CH}^i(N(X)).$ We show that for divisors, the image of $N$ can be
describe much easier, and that in this case $N$ is injective.
Then we show that Mumford's result, that Chow groups
are not finite dimensional, implies that in general the morphism
$N:CH^i(X)\rightarrow\str{CH}^i(N(X))$ is not injective.
After that we show that our notion of complexity behaves well under
the intersection product of cycles, and we show that rational
equivalence behaves somehow bad under this notion of complexity.

\vspace{\abstand}

Let $K$ be a field and $X\hookrightarrow\P_K^n$ be a closed
immersion. Then we define a notion of \emph{complexity} on the Chow ring of $X$ as follows:

\vspace{\abstand}

\begin{defi}
An element $x\in CH^i(X)$ has \emph{complexity less than $c\in\N$} if we can
write $x$ as
$$x=\sum_{i=1}^{n} \alpha_i [X_i]$$
where $|\alpha_i|<c$, $n<c$ and the $X_i\hookrightarrow X$ are integral subschemes
of degree $<c$. We define the notion of \emph{*complexity less than $c\in\Ns$} for
*cycles on *projective varieties in the obvious analogue way.
\end{defi}

\vspace{\abstand}

The next lemma shows that this notion of complexity is quite natural,
if we want to understand the image of $N$ for Chow groups.

\vspace{\abstand}

\begin{lemma}\label{imageN}
Let $K$ be an internal field and $X\hookrightarrow\P^n_K$ a closed
immersion. Then a cycle $x'\in\str{CH}^i(N(X))$ is in the image of
the morphism
$$N:CH^i(X)\rightarrow \str{CH}^i(N(X))$$
if and only if the *complexity of $x'$ is less than $d$ for a
$d\in\N\subset\Ns$.
\end{lemma}
\begin{proof}
An element $x=\sum \alpha_i[X_i]$ is mapped to $N(x)=\sum \alpha_i
[N(X_i)]$. By \cite{enlsch}[Cor. 6.18] the *degree of $N(X_i)$ is
equal to the degree of $X_i$. Furthermore, by \cite{enlsch}[Corollary 6.21], a prime cycle
$Y_i\hookrightarrow N(X)$ is of the form $N(X_i)$ for an
$X_i\hookrightarrow X$ if and only if its *degree is in $\N\subset\str{\N}$.
\end{proof}

\vspace{\abstand}

For cycles of codimension one the situation is much easier, and we can simply use the Hilbert polynomial
instead of the notion of complexity.

\begin{thm}
Let $K$ be an internal field, $X$ a projective $K$-scheme with
integral geometric fiber and $\phi\in\Q[t]$ a rational
polynomial. Then the morphism
$$Pic^{\phi}(X)\rightarrow \str{Pic}^{\phi}(N(X))$$
is bijective.
\end{thm}

\begin{proof}
We find a subring $A_0\subset K$ of finite type over $\Z$ and a
projective $A_0$-scheme $\fX$ with geometrically integral fibers such
that $\fX_0\otimes_{A_0}K=X$. We denote by $\underline{Pic}_{\fX_0/A_0}^{\phi}$
the relative Picard-functor. By  \cite{SGA6}[XIII, 3.2 (iii) and 2.11]
$\underline{Pic}_{\fX_0/A_0}^{\phi}$ is representable by a scheme
$Pic_{\fX_0/A_0}^\phi$ of finite type over $A_0$. Then
$\underline{Pic}_{X/K}^\phi$ is represented by
$Pic_{\fX_0/A_0}^\phi\otimes_{A_0}K,$ and $\str{\underline{Pic}}_{N(X)/k}^\phi$
is represented by
$$\str{Pic}_{\fX_0/A_0}^\phi\otimes_{A_0}K=N(Pic_{\fX_0/A_0}^\phi\otimes_{A_0}K),$$
so the theorem follows from \cite{enlsch}[theorem 4.14].
\end{proof}

\vspace{\abstand}

\begin{cor}\label{cor:degree}
Let $X$ be as in the above theorem. Then the morphism
$$N: CH^1(X)\rightarrow \str{CH}^1(N(X))$$
is injective, and the image consists of those *divisors whose
Hilbert polynomial is in $\Q[t]\subset\str{\Q}\str{[t]}$.
\end{cor}

\vspace{\abstand}

Lemma \ref{imageN} describes the image of $N:CH^i(X)\rightarrow
\str{CH}^i(N(X)).$ Now we want to show that in general $N$ fails to be
injective.

For that we consider a smooth projective irreducible surface $X$ over
$\C$ with $H^2(X,\cO_X)\neq 0$. For such a surface Mumford showed:

\vspace{\abstand}

\begin{thm}\label{theorem:mumford}
Let $X$ be as above. Then $CH^2(X)_0:=\{\text{0-cycles of
  degree zero}\}$ is not finite dimensional, i.e.
for all $n\in\N$ the natural map
$$S^n X(\C)\times S^n X(\C)\rightarrow CH^2(X)_0$$
is not surjective.
\end{thm}
\begin{proof}
That is the main result of \cite{mumfordfinite}.
\end{proof}

\vspace{\abstand}

We have further the following characterisation of finite
dimensionality:

\vspace{\abstand}

\begin{satz}\label{prop:jannsen}
Let $X$ be a smooth, projective, geometrically irreducible variety of
dimension $d$ over a field $k$, and let $\Omega\supseteq k$ be an
algebraically closed field. Consider the following statements
\begin{enumerate}
\item There is an $n\in\N$ such that
  $$S^n X(\Omega)\times S^n X(\Omega)\rightarrow
  CH^d(X_{\Omega})_0$$ is surjective
\item If $B\subseteq X_{\Omega}$ is a smooth linear space section of
  dimension one, then
  $$ CH^d(X_{\Omega}\setminus B)=0$$
\item The canonical map
  $$ CH^d(X_{\Omega})_0\rightarrow Alb(X)(\Omega)$$
  is an isomorphism, where $Alb(X)$ is the albanese variety of $X$.
\end{enumerate}
Then $(iii)\Rightarrow (ii)\Rightarrow (i)$. If $\Omega$ is
uncountable, all three statements are equivalent.
\end{satz}
\begin{proof}
See \cite{JaMotivesBand}[1.6].
\end{proof}

\vspace{\abstand}

Furthermore, For the algebraic closure of a finite field we have the
following result:

\begin{thm}\label{theorm:KS}
Let $X$ be as in the above proposition, where $k$ is now a finite
field. Then the morphism
$$CH^d(\bar{X})_0\rightarrow Alb(X)(\bar{k})$$
is an isomorphism. Here $Alb(X)$ denotes the albanese variety of a
scheme $X$.
\end{thm}
\begin{proof}
See \cite{KatoSaito83}[§9]
\end{proof}

\vspace{\abstand}

Now let $P\in\str{\P}-\P$ be an infinite prime with
$\bar{\Q}\subset\str{\F}_P$, and let $k$ be an *algebraic closure of
$\str{\F}_P$. Then $k$ is in particular an algebraically closed field, and $\C$ can be
embedded in $k$.

We can then use Theorem \ref{theorem:mumford} \ref{theorm:KS} to prove the following theorem.

\vspace{\abstand}

\begin{thm}
Let $X$ be as above. Then the morphism
$$CH^{2}(X_k)\rightarrow \str{CH}^{2}(N(X_k))$$
is not injective.
\end{thm}
\begin{proof}
From \ref{theorem:mumford} and \ref{prop:jannsen} it follows
that for a smooth
linear space section $B\hookrightarrow X_k$ of dimension one we have
$$CH^{2}(X_k\setminus B)\neq 0.$$
But $k$ is internally the *algebraic closure of a *finite
field. Therefore by transfer of \ref{theorm:KS}, the map
$$\str{CH}^{2}(N(X_k))_0\rightarrow \str{Alb}(N(X))(k)$$
is an isomorphism. Then by \ref{prop:jannsen} again we have
$$\str{CH}^{2}(N(X_k)\setminus N(B))=0.$$
Now we consider the commutative diagram
\[
\xymatrix@C=20mm{
CH^1(B) \ar[r]^{N} \ar[d]^{i_*} & \str{CH}^1(N(B)) \ar[d]^{N(i)_*} \\
CH^2(X) \ar[r]^{N} \ar[d]^{j^*} & \str{CH}^2(N(X)) \ar[d]^{N(j)^*}\\
CH^2(X\setminus B)\ar[r]^N \ar[d] & \str{CH}^2(N(X)\setminus N(B)) \ar[d]^{\sim}\\
0 &  0
}
\]

which has exact columns and where $j:(X\setminus B)\hookrightarrow X$ denotes the open
immersion. Now let $x\in CH^2(X)$ with $j^* x\neq 0$. Then by diagram chasing
there is a $y\in \str{CH}^1(N(B))$ such
that $N(i_*)y=N(x)$. Now we have $\deg(y)=\deg(N(x))=\deg(x)$. Therefore
by \ref{cor:degree} there is an $\tilde{y} \in CH^1(B)$
with $N(\tilde{y})=y$. But then we have $x-i_*y\neq 0$ but
$N(x-i_*y)=0.$
\end{proof}

\vspace{\abstand}

On the other hand, the above cited results can also be used to show surjectivity
of $N$ in another situation:

\begin{thm}
Let $k$ be the *algebraic closure of a *finite field, and let $X/k$ a
smooth, projective and geometrically irreducible scheme of dimension
$d$. Then the morphism
$$N:CH^d(X)_0\rightarrow \str{CH}^d(N(X))_0$$
is surjective.
\end{thm}

\begin{proof}
By transfer of \ref{theorm:KS} and \ref{prop:jannsen} we have the
following diagram
\[
\xymatrix{
\str{CH}^1(NB)_0 \ar[r]^{(Ni)_*} & \str{CH}^d(N(X))_0 \ar[r] &
\str{CH}^d(N(X)\setminus N(B)) =0 \\
CH^1(B)_0 \ar[u] \ar[r] & CH^d(X) \ar[u]
}
\]
with exact first line. But the vertical map on the left is an
isomorphism by \ref{cor:degree}, and so the claim follows.
\end{proof}

\vspace{\abstand}

In contrast to that we have the following result for surfaces over $\str{\C}$:

\begin{satz}
Let $X$ be a smooth, projective, and irreducible surface over $\C$ with $H^2(X,\cO_X)\neq 0$.
Then the map
$$ CH_0(X_{\str{\C}})_0\rightarrow \str{CH}_0(N(X_{\str{\C}}))_0$$
is not surjective.
\end{satz}

\begin{proof}
By Theorem \ref{theorem:mumford} and its transfer we have:

$$\forall n\in\N:S^nX_{\str{\C}}(\str{\C})\times S^nX_{\str{\C}}(\str{\C})\rightarrow CH_0(X_{\str{\C}})_0 \text{ is not surjective.}$$
and
$$\forall n\in\str{\N}:S^nN(X_{\str{\C}})(\str{\C})\times S^nN(X_{\str{\C}})(\str{\C})\rightarrow \str{CH}_0(N(X_{\str{\C}}))_0
\text{ is not surjective,}$$

and we have the following commutative diagram:
$$
\xymatrix{
S^nX_{\str{\C}}(\str{\C})\times S^nX_{\str{\C}}(\str{\C}) \ar[r] \ar[d] & CH_0(X_{\str{\C}})_0 \ar[d] \\
S^nN(X_{\str{\C}})(\str{\C})\times S^nN(X_{\str{\C}})(\str{\C}) \ar[r] & \str{CH}_0(N(X_{\str{\C}}))_0
}
$$
Now $\cup_{n\in\N}(\im(S^nX_{\str{\C}}(\str{\C})\times S^nX_{\str{\C}}(\str{\C})\rightarrow CH_0(X_{\str{\C}})_0)=CH_0(X_{\str{\C}})_0$,
and if $CH_0(X_{\str{\C}})_0\rightarrow \str{CH}_0(N(X_{\str{\C}}))_0$ was surjective, then for all $n\in\str{\N}-\N$ the map
$$S^nN(X_{\str{\C}}) \times S^nN(X_{\str{\C}})\rightarrow \str{CH}_0(N(X))_0$$
would be surjective.
\end{proof}

\vspace{\abstand}
Now we want to show that the non-injectivity of the morphism $N$ somehow says
that our notion of complexity does not go along very well with
rational equivalence. For that let us remind you first another,
similar situation.

Let $f_1$, $f_2$ and $g$ be polynomials over a field $k$ such that $g\in
(f_1,f_2)$. That means that there are polynomials $a_1$ and $a_2$ with
\begin{equation}\label{vorueberl}
g=a_1\cdot f_1 +a_2\cdot f_2.
\end{equation}
Now one can ask for a bound of the minimal degree of $a_1$ and $a_2$ in terms
of the degrees of $f_1$, $f_2$ and $g$ such that (\ref{vorueberl}) holds. And in
fact such a bound exists, and it only depend on the degrees of
$f_1$, $f_2$ and $g$, but not on their coefficients or the field $k$ (!). This
follows for example from the nonstandard fact that the ring
homomorphism
$$k[x_1,\cdots,x_n]\rightarrow k\str{[}x_1,\cdots,x_n]$$
is flat for an internal field $k$ (cf. \cite{vandendries}).

Now we want to answer a similar question for the rational equivalence
of cycles. For that let $x,y\in Z_k(X)$ be two cycles, which are
rational equivalent. That means that there is a cycle $z\in
Z_{k+1}(X\times \P^1)$ such that
$x-y=z_{|X\times\{0\}}-z_{|X\times\{1\}}$. Now a natural question is
wether we can bind the complexity of $z$ by the complexities of $x$ and
$y$. One could hope that the following was true:
\begin{quote}
For all $d,n\in\N$ there is a  constant $C(n,d)\in\N$ such that for
all fields $K$, all closed subschemes $X\hookrightarrow\P^n_K$ of
degree less than $d$ and all rational equivalent cycles $x,y\in
Z_k(X)$ with complexity less than $d$, there is a cycle
$z\in Z_{k+1}(X\times\P^1_K)$ with complexity less than $C(n,d)$, such that
$$x-y=z_{|X\times\{0\}}-z_{|X\times\{1\}}.$$
\end{quote}
But this can not be true:

\vspace{\abstand}

\begin{thm}\label{thmfalse}
The above statement is false.
\end{thm}

\begin{proof}
If the statement were true, then $N:CH_k(X)\rightarrow CH_k(N(X))$
would be injective by Lemma \ref{imageN}.
\end{proof}

\vspace{\abstand}

Next we want to see how our notion of complexity of a cycles
behaves under intersection products.
If we consider two cycles $x\in CH^i(X)$ and $y\in CH^j(X)$, both
of complexity less than $d$, it is natural to ask about the complexity of their
intersection product $x\cdot y \in CH^{i+j}(X)$. If we could write
$x=\sum_{i=1}^n \alpha_i [X_i]$ and $y=\sum_{j=1}^m \beta_j [Y_j]$
with $\alpha_i$, $\beta_j$, $n$, $m$, $\deg(X_i)$ ,$\deg(Y_j)<d$ and such that all $X_i$ and
$Y_j$ intersect properly, the complexity of $x\cdot y$ would be
$<d^4$. But in general one has to move the cycles in their rational
equivalence class to get proper intersections, and it is hard to
control the complexities during this process (cf. \ref{thmfalse}.
But using the previous result,
we can at least prove the existence of a uniform bound for the complexity of the product:

\vspace{\abstand}

\begin{thm}
For all $d,n\in\N$ there is a constant $C(d,n)$ with the following
property: For all fields $k$, all closed subschemes
$X\hookrightarrow\P^n_k$ of degree less than $d$, and cycles $x\in
CH^i(X)$ and $y\in CH^j(X)$, both of complexity
less than $d$, the product $x\cdot y\in CH^{i+j}(X)$ is of complexity
less than $C(d,n)$.
\end{thm}
\begin{proof}
All statements are about schemes or subschemes which are of finite
presentation over a field, and the intersection product behaves well
under field extension.  Therefore it is enough to consider all
fields which are finitely generated over their prime fields, and we
can choose a category of rings $\cR$ which contains all such fields.
Now we assume that the statement is false. Then by transfer there are
an internal field $K\in\cR$, a *closed subscheme
$X'\hookrightarrow\str\P^n_K$ of degree less than $d$ and *cycles
$x'\in\str{CH}^i(X)$ and $y'\in\str{CH}^j(X)$, both of complexity less than $d$,
such that the product $x'\cdot y' \in\str{CH}^{i+j}(X')$ is not of
complexity less than $n$ for all $n\in\N$. But because of our assumption about
the degree of $X'$ and the complexities of $x$ and $y$, there are a closed subscheme
$X\in\P^n_K$ and cycles $x\in CH^i(X)$ and $y\in CH(Y)$ with
$$N(X)=X', N(x)=x'\text{ and } N(y)=y'.$$
Furthermore, by \ref{de:NfMoCoh} we have $N(x\cdot y)=x'\cdot y'$. But the
complexity of $x\cdot y$ is less than $n_0$ for an $n_0\in\N$, and then
the complexity of $N(x\cdot y)$ is also less than $n_0\in\N$, a
contradiction.
\end{proof}
\vspace{\abstand}
\begin{bem}
This theorem corresponds to the fact that the intersection product is \emph{constructible},
proven in \cite{macintyre} by careful analysis of the construction given in \cite{fulton}.
\end{bem}

\vspace{\abstand}

\begin{bem}
With the same argument, a similar result can be shown for higher Chow groups.
\end{bem}

\vspace{\abstand}

\vspace{\abstand}

\appendix
\section{Lengths and the Koszul complex}

\bigskip

\noindent
Let $k$ be an internal field,
let $A$ be a $k$-algebra of finite type,
let $\p\subseteq A$ be a prime ideal,
and let $M$ be a finitely generated $A$-module.

\vspace{\abstand}

\begin{lemma}\label{lemmaModLokZero}
  $M_\p=0\Longrightarrow[\s{M}]_{\s{\p}}=0$.
\end{lemma}

\begin{proof}
  \[
    M_\p=0\Longrightarrow
    \exists f\in A\setminus\p:\ fM=0\Longrightarrow
    f[\s{M}]=0\Longrightarrow
    [\s{M}]_{\s{\p}}=0.
  \]
\end{proof}

\vspace{\abstand}

\begin{lemma}\label{lemmaModLokFF}
  Let $\varphi:M'\longrightarrow M$ be a morphism of finitely generated
  $A$-modules such that $\kernel{\varphi_\p}=0$. Then
  \begin{enumerate}
    \item\label{lMLFFi}
      $\kernel{[\s{\varphi}]_{\s{\p}}}=0$ and
    \item\label{lMLFFii}
      $\kappa(\p)\cong\coker{\varphi_\p}\Longrightarrow
      \str{\kappa(\s{\p})}\cong\coker{[\s{\varphi}]_{\s{\p}}}$.
  \end{enumerate}
\end{lemma}

\begin{proof}
  \[
    \kernel{\varphi_\p}=0\Longrightarrow
    \exists f\in A\setminus\p:\ \kernel{\varphi_f}=0
    \stackrel{\mbox{\tiny{$\s{}$ exact}}}{\Longrightarrow}
    \kernel{[\s{\varphi}]_f}=0\Longrightarrow
    \kernel{[\s{\varphi}]_{\s{\p}}}=0,
  \]
  which proves \ref{lMLFFi}.
  For \ref{lMLFFii}, assume $\kappa(\p)\cong\coker{\varphi_p}$.
  Then
  \begin{multline*}
    \exists f\in A\setminus\p:\ A_f/\p A_f\cong\coker{\varphi_f}
    \stackrel{\mbox{\tiny{$\s{}$ exact}}}{\Longrightarrow} \\
    [\s{A}]_f/\p[\s{A}]_f\cong\coker{[\s{\varphi}]_f}\Longrightarrow
    \str{\kappa(\s{\p})}\cong\coker{[\s{\varphi}]_{\s{\p}}},
  \end{multline*}
  where we use $\s{\p}=\p[\s{A}]$ (\cite[2.5]{vandendries}).
\end{proof}

\vspace{\abstand}

\begin{lemma}\label{lemmaLength}
  $\length{A_\p}{M_p}=\slength{[\s{A}]_{\s{\p}}}{[\s{M}]_{\s{\p}}}
  \in\Nn\amalg\{\infty\}$.
\end{lemma}

\begin{proof}
  Put $l:=\length{A_\p}{M_\p}$ and
  $\str{l}:=\slength{[\s{A}]_{\s{\p}}}{[\s{M}]_{\s{\p}}}$.
  First consider the case $l<\infty$.
  Then by the definition of "length", there exists a chain
  \[
    0=M_0\subseteq
    M_1\subseteq
    M_2\subseteq
    \ldots\subseteq
    M_{l-1}\subseteq
    M_l=M_\p
  \]
  of (finitely generated) $A_\p$-modules
  satisfying $\kappa(\p)\cong M_i/M_{i-1}$ for $i=1,\ldots, l$.
  Since the $M_i$ are finitely generated, there exist an $f\in A\setminus\p$
  and a tower
  \[
    0=M'_0\xrightarrow{\varphi_0}
    M'_1\xrightarrow{\varphi_1}
    M'_2\xrightarrow{\varphi_2}
    \ldots\xrightarrow{\varphi_{l-2}}
    M'_{l-1}\xrightarrow{\varphi_{l-1}}
    M'_l=M_f
  \]
  of (finitely generated) $A_f$-modules with $(M'_i)_\p\cong M_i$
  for $i=0,\ldots,l$.
  Applying \ref{lemmaModLokZero} and \ref{lemmaModLokFF}
  to $A_f$ instead of $A$,
  it follows that we have a chain
  \[
    0=\tilde{M_0}\subseteq
    \tilde{M_1}\subseteq
    \tilde{M_2}\subseteq
    \ldots\subseteq
    \tilde{M_{l-1}}\subseteq
    \tilde{M_l}=M_\p
  \]
  of *finitely generated $[\s{A}]_{\s{\p}}$-modules,
  where $\tilde{M_i}:=[\s{M'_i}]_{\s{\p}}$ for $i=0,\ldots,l$.
  Furthermore,
  $\str{\kappa}(\s{\p})\cong\tilde{M}_i/\tilde{M}_{i-1}$ for $i=1,\ldots, l$,
  which proves $\str{l}=l$.

  Now let $M_\p$ have infinite length, and choose an arbitrary $n\in\Np$.
  Since $M_\p$ has infinite length, there exists an $A_\p$-submodule
  $M'$ of $M_\p$ of length $n$, and $M'\cong M''_\p$ for a suitable
  $f\in A\setminus\p$ and an $A_f$-module $M''$.
  By what has already been proven (taking $A_f$ instead of $A$),
  $\slength{[\s{A}]_{\s{\p}}}{[\s{M''}]_{\s{\p}}}=n$, so
  $[\s{M}]_{\s{\p}}$ contains a *submodule of *length $n$.
  Since $n$ was chosen arbitrarily, $\str{l}\geq n$ for all $n\in\Np$,
  so $\str{l}$ lies in $(\Nns\setminus\Nn)\amalg\{\str{\infty}\}$
  and is hence \emph{infinite}.
\end{proof}

\vspace{\abstand}

Now we consider an element $m\in M$ and the Koszul complex
$$\mathcal{K}(M,m):= 0\rightarrow A \rightarrow M \rightarrow \Lambda^2
M \rightarrow \dots$$
In fact this is a bounded complex of finitely generated A-modules.
\begin{lemma}\label{le:NcomWithKoszul}
$$N(\mathcal{K}(M,m))=\str{\mathcal{K}}(N(M),N(m))$$
\end{lemma}
\begin{proof}
By transfer we have
$$\str{\mathcal{K}}(N(M),N(m)):= 0\rightarrow N(A) \rightarrow N(M) \rightarrow \Lambda^2
N(M) \rightarrow \dots$$
In \cite{enlsch}[5.14] it is shown that $N$ for modules commutes with
tensor products. Similarly one can show that $N$ commutes with
alternating products. So the claim follows.
\end{proof}
\vspace{\abstand}


\bibliographystyle{alpha}
\bibliography{../Literatur}

\end{document}